\documentclass[11pt]{article}
\long\def\symbolfootnote[#1]#2{\begingroup\def\thefootnote{\fnsymbol{footnote}}
\footnote[#1]{#2}\endgroup}
\usepackage{amsmath,amssymb,amsthm,graphics}
\usepackage{cite}
\begin{document}

\newcommand\blfootnote[1]{%
  \begingroup
  \renewcommand\thefootnote{}\footnote{#1}%
  \addtocounter{footnote}{-1}%
  \endgroup
}

\newtheorem{lemma}{Lemma}[section]
\newtheorem{theorem}[lemma]{Theorem}
\newtheorem{corollary}[lemma]{Corollary}
\newtheorem{conjecture}[lemma]{Conjecture}
\newtheorem{prop}[lemma]{Proposition}

\newtheorem*{T1}{Theorem 1}
\newtheorem*{T2}{Theorem 2}
\newtheorem*{T3}{Theorem 3}
\newtheorem*{C2}{Corollary 2}
\newtheorem*{C3}{Corollary 3}
\newtheorem*{T4}{Theorem 4}
\newtheorem*{C5}{Corollary 5}
\theoremstyle{remark}
\newtheorem{remark}[lemma]{Remark}
\theoremstyle{definition}
\newtheorem{defn}[lemma]{Definition}
\newtheorem{question}[lemma]{Question}
\renewcommand{\labelenumi}{(\roman{enumi})}
\newcommand{\Ext}{\mathrm{Ext}}
\newcommand{\Hom}{\mathrm{Hom}}
\newcommand{\ch}{\mathrm{ch}}
\newcommand{\Soc}{\mathrm{Soc}}

\newenvironment{changemargin}[1]{%
  \begin{list}{}{%
    \setlength{\topsep}{0pt}%
    \setlength{\topmargin}{#1}%
    \setlength{\listparindent}{\parindent}%
    \setlength{\itemindent}{\parindent}%
    \setlength{\parsep}{\parskip}%
  }%
  \item[]}{\end{list}}

\parindent=0pt
\addtolength{\parskip}{0.5\baselineskip}
\title{The second cohomology of simple $SL_3$-modules}
\author{David I. Stewart}
\date{New College, Oxford}
\maketitle
{\small {\bf Abstract.}  Let $G$ be the simple, simply connected algebraic group $SL_3$ defined over an algebraically closed field $K$ of characteristic $p>0$. In this paper, we find $H^2(G,V)$ for any irreducible $G$-module $V$. When $p>7$ we also find $H^2(G(q),V)$ for any irreducible $G(q)$-module $V$ for the finite Chevalley groups $G(q)=SL(3,q)$ where $q$ is a power of $p$.} 
\section{Introduction}
Let $G=SL_3$ be defined over an algebraically closed field $K$ of characteristic $p>0$. Let $T$ be a maximal torus of $G$. Recall that the weight lattice $X(T)$ of $T$ is generated over $\mathbb Z$ by the fundamental weights $\lambda_1$ and $\lambda_2$ so can be identified with $\mathbb Z\times \mathbb Z$ by $(a,b)=a\lambda_1+b\lambda_2$, while the dominant weights $X_+$ of $T$ can be identified with $(a,b)$ such that $a,b\in\mathbb Z_{\geq 0}$. Thus for each pair of positive integers $(a,b)$ there is an irreducible module $L(a,b)$ of highest weight $(a,b)$. Let $V$ be a $G$-module. As $G$ is defined over $\mathbb F_p$ we have the notion of the $d$th Frobenius map which raises each matrix entry in $G$ to the power $p^d$. When composed with the representation $G\to GL(V)$, this induces the twist $V^{[d]}$ of $V$. The first Frobenius map has a infinitesimal scheme-theoretic kernel $G_1\triangleleft G$, (see for instance, \cite[I.9]{Jan03}).

Let $(a_0,b_0)+p(a_1,b_1)+p^2(a_2,b_2)+\dots$ be the $p$-adic expansion of the pair of integers $(a,b)$. By Steinberg's tensor product theorem, the irreducible module $L(a,b)$ of high weight $(a,b)$ is given by \[L(a_0,b_0)\otimes L(a_1,b_1)^{[1]}\otimes L(a_2,b_2)^{[2]}\otimes\dots\ .\]

In \cite{Ste10} we calculated the second cohomology of simple modules for $SL_2$. We perform the task here for $SL_3$-modules.

\begin{T1}\label{Thm1} Let $V=L(a,b)^{[d]}$ be any Frobenius twist (possibly trivial) of the irreducible $G$-module $L(a,b)$ with highest weight $(a,b)$. For $p\geq 3$, let $L(a,b)$ or $L(b,a)$ be one of \begin{align*}
&L(1,1)^{[1]},\\
&L(p-3,0)\otimes L(0,1)^{[1]},\\
&L(p-2,1)\otimes L(p-3,p-2)^{[1],}\\
&L(p-2,1)\otimes L(2,p-3)^{[1]}\otimes L(1,0)^{[2]},\\
&L(p-2,1)\otimes L(p-2,2)^{[1]}\otimes L(0,1)^{[2]},\\
&L(p-2,1)\otimes L(0,1)^{[1]}\otimes L(p-2,p-2)^{[r+1]},\\
&L(p-2,1)\otimes L(0,1)^{[1]}\otimes L(p-2,1)^{[r+1]}\otimes L(0,1)^{[r+2]},\\
&L(p-2,1)\otimes L(0,1)^{[1]}\otimes L(1,p-2)^{[r+1]}\otimes L(1,0)^{[r+2]},\\
&L(p-2,p-2)\otimes L(p-2,p-2)^{[r]},\text{ or}\\
&L(p-2,p-2)\otimes L(1,p-2)^{[r]}\otimes L(1,0)^{[r+1]}
\end{align*}
for any $r>0$; and for $p=2$, let $L(a,b)$ or $L(b,a)$ be one of \begin{align*}
&L(1,1)^{[1]},\\
&L(1,0)\otimes L(0,1)^{[2]},\ (^*)\\
&L(1,0)\otimes L(1,0)^{[1]}\otimes L(0,1)^{[2]},\\
&L(1,0)\otimes L(1,0)^{[1]}\otimes L(0,1)^{[r+1]}\otimes L(0,1)^{[r+2]},\\
&L(1,0)\otimes L(1,0)^{[1]}\otimes L(1,0)^{[r+1]}\otimes L(1,0)^{[r+2]}\end{align*}\blfootnote{(*) In the published version of this article, the value $2$ appears incorrectly as $1$. This was due to an incorrect reading of value of $H^1(B_1,(1,0))$ from \cite[Appendix C]{Wri11}. Tributary results and proofs have been adjusted accordingly. We thank Adam Thomas for noticing this error.}
for any $r>0$.

Then $H^2(G,V)=K$. For all other irreducible $G$-modules $V$, $H^2(G,V)=0$.
 \end{T1}

This result was inspired by the methods of G. McNinch's paper \cite{McN02} which computes $H^2(G,V)$ for simply connected algebraic groups $G$ acting on modules $V$ with $\dim V\leq p$. Note that this paper makes a correction to Theorem A in a special case: if $G=SL_3(K)$, $p=3$ and  $V=L(1,0)^{[1]}$ or $L(0,1)^{[1]}$ is a Frobenius twist of the $3$-dimensional natural module for $G$ or its dual then we show that $H^2(G,W)=K$ for $W$ any Frobenius twist of $V$.

Finally, for $q=p^r$, let $G(q)$ denote the $\mathbb F_q$-points of the algebraic group $G$. Then $G(q)=SL(3,q)$ is a finite group of Lie type.

Using work of Bendel, Nakano and Pillen we also obtain the second cohomology of the finite groups $SL(3,q)$ with coefficients in simple modules when $p>7$. This is our Theorem 2 and the subject of \S3. We conclude the paper with some discussion of the above results, and ask a few questions.

{\bf Acknowledgements.} Some of the work in this paper was prepared towards the author's PhD qualification under the supervision of Prof. M. W. Liebeck, with financial support from the EPSRC. We would like to thank Prof. Liebeck for his help in producing this paper. Also we thank the anonymous reviewer for many helpful suggestions and corrections; and the editor P. Tiep for his suggestion to extend the work to cover the finite groups of Lie type.

\section{Proof of Theorem 1}
We begin with a little notation which we keep compatible with \cite{Jan03}:

Let $B$ be a Borel subgroup of a reductive algebraic group $G$, containing a maximal torus $T$ of $G$. Recall that for each dominant weight $\lambda\in X(T)$ for $G$, the space $H^0(\lambda):=H^0(G/B,\lambda)=\mathrm{Ind}_B^G(\lambda)$ is a $G$-module with highest weight $\lambda$ and with socle $\mathrm{Soc}_G H^0(\lambda)=L(\lambda)$, the irreducible $G$-module of highest weight $\lambda$. The Weyl module of highest weight $\lambda$ is $V(\lambda)\cong H^0(-w_0\lambda)^*$ where $w_0$ is the longest element in the Weyl group. For $G=SL_3$, we identify $X(T)$ with $\mathbb Z^2$. In this case  $H^0(a,b)=V(a,b)^*=V(b,a)$. When $0\leq a,b<p$, we say that $(a,b)$ is a restricted weight and we write $(a,b)\in X_1$. Let $(a,b)=(a_0,b_0) +p(a_1,b_1)+\dots+p^n (a_n,b_n)$. We will often need to refer to specific parts of this $p$-adic expansion, so we write 
\begin{align*}(a,b)&=(a_0,b_0)+ p(a,b)'\\&=(a_0,b_0)+p(a_1,b_1)+p^2(a,b)'',\end{align*} where $(a,b)'=(a_1,b_1)+p(a_2,b_2)+\dots+ p^{n-1}(a_n,b_n)$ and $(a,b)''=(a_2,b_2)+p(a_3,b_3)+\dots+p^{n-2} (a_n,b_n)$.

For $\mathbb ZR$, the root lattice of $G$, and $W$ the Weyl group of $G$, recall the dot action of the affine Weyl group $W_p=W\rtimes p\mathbb Z R$ on the weight lattice $X(T)$: where $W\leq W_p$ acts as $w\centerdot \lambda=w(\lambda+\rho)-\rho$, where $\rho$ is the half-sum of the positive roots and $\mathbb ZR$ acts by translations. In case $G=SL_3$, $\rho=(1,1)$. We repeatedly use the linkage principle for $G$ and $G_1$ (see \cite[II.6.17]{Jan03} and \cite[II.9.19]{Jan03}); this means  that if $\Ext_G^i(L(\lambda),L(\mu))\neq 0$ or $\Ext_{G_1}^i(L(\lambda),L(\mu))\neq 0$ for any $i\geq 0$ then $\lambda\in W_p\centerdot \mu$ or $\lambda\in W_p\centerdot \mu+pX(T)$ respectively. In Table 1 we detail the cases that can occur for $SL_3$ for a restricted weight $(a_0,b_0)$, where $s_\alpha$ ($s_\beta$, respectively) denotes the reflection in the hyperplane perpendicular to the simple root $\alpha$ corresponding to the fundamental weight $(1,0)$ ($(0,1)$ respectively).

\begin{table}[h]
\begin{center}
\begin{tabular}{|c|c|c|c|c|c|}
\hline
$w$ & $l(w)$ & $w.(a_0,b_0)$\\
\hline
$1$ & $0$ & $(a_0,b_0)$\\
$s_\alpha$ & $1$ & $(-a_0-2,a_0+b_0+1)$\\
$s_\beta$ & $1$ & $(a_0+b_0+1,-b_0-2)$\\
$s_\beta s_\alpha$ & $2$ & $(b_0,-a_0-b_0-3)$\\
$s_\alpha s_\beta$ & $2$ & $(-a_0-b_0-3,a_0)$\\
$w_0$ & $3$ & $(-b_0-2,-a_0-2)$\\
\hline\end{tabular}\caption{Dot actions}\end{center}\end{table}\label{dots}

In particular, taking $(a_0,b_0)=(0,0)$ we record the following lemma for later use:

\begin{lemma}\label{linkages} The only restricted weights $G_1$-linked to $(0,0)$ up to duals are $(0,0),\ (p-2,1),\ (p-3,0),\ (p-2,p-2)$, and the only restricted weight $G$-linked to $(0,0)$ is $(p-2,p-2)$.\end{lemma}

Let $V$ be a $G$-module. Recall the kernel $G_1$ of the Frobenius map $F:G\to G$ is a scheme-theoretic normal subgroup of $G$. Thus the cohomology group $M=H^i(G_1,V)$ has the structure of a $G/G_1$-module, and so also of a $G$-module. Since $G_1$ acts trivially on this module, there is a Frobenius untwist $M^{[-1]}$ of $M$.  By \cite[I.9.5]{Jan03}, $G/G_1\cong F(G)$, where $F$ is the first Frobenius morphism. Thus $G/G_1$ acts on $H^i(G_1,V)$ as $G$ acts on $H^i(G_1,V)^{[-1]}$.

There are two main ingredients in the proof of Theorem 1. The first is the Lyndon-Hochschild-Serre spectral sequence \cite[6.6 (3)]{Jan03} applied to $G_1\triangleleft G$, using the observations in the preceding paragraph.

\begin{prop}\label{spec} There is for each $G$-module $V$ a spectral sequence \[E_2^{nm}=H^n(G,H^m(G_1,V)^{[-1]})\Rightarrow H^{n+m}(G,V).\]\end{prop}

We will always refer by $E_*^{**}$ to terms in the above spectral sequence. We briefly recall the important features of spectral sequences for the unfamiliar reader. The lower subscript refers to the sheet of the spectral sequence. Only the second sheet is explicitly defined in this example. The point $E_2^{nm}$ is defined only for the first quadrant, i.e. for $n,m\geq 0$; for any other $n,m$ we have $E_2^{nm}=0$. On the $i$th sheet, maps in the spectral sequence through the $n,m$th point go \[\dots\to E_i^{n-i,m+i-1}\stackrel{\rho}{\to} E_i^{nm}\stackrel{\sigma}{\to}E_i^{n+i,m-i+1}\to\dots.\] These form a complex, i.e. $\sigma\rho=0$. One then gets the point of the next sheet, $E_{i+1}^{nm}$ as a section of $E_i^{nm}$ by setting $E_{i+1}^{nm}=$ker $\sigma/$im $\rho$, i.e. the cohomology at that point. If $i$ is big enough, maps go from outside the defined quadrant to a given point and then from that point to outside the defined quadrant again. Thus each point of the spectral sequence eventually stabilises. We denote the stable value by $E_\infty^{nm}$. Finally, one gets \begin{equation}\label{sumup}H^r(G,V)=\bigoplus_{n+m=r}E_\infty^{nm}\end{equation} explaining the notation `$\Rightarrow H^{n+m}(G,V)$'.

For the remainder of Section 2 we take $G=SL_3$. The second main ingredient in our calculation is the structure of $H^i(G_1,L(a,b))^{[-1]}$ as a $G$-module for $i\leq 2$. 

When $i=1$ this is a special case of \cite[3.3.2]{Yeh82}. Since this useful result has not appeared in print, we record it here.
\begin{prop}[Yehia]\label{yehia} For every $\lambda,\mu \in X_1$ the $G$-structure of $E=\Ext^1_{G_1}(L(\mu),L(\lambda))$ when $E\neq 0$ is given in the following table:
{\footnotesize \begin{center}\begin{tabular}{| l | l | l l |}\hline
$\lambda$ & $\mu$ & $E$&\\\hline
$(r,s)\in C_0$ & $(p-s-2,p-r-2)$ & $K$ & $[K\oplus L(1,0)\oplus L(0,1),\ p=3]$\\
 & $(r+s+1,p-s-2)$ & $L(1,0)$ & $[K\oplus L(1,0)\oplus L(0,1),\ p=3]$\\
 & $(p-r-2,r+s+1)$ & $L(0,1)$ & $[K\oplus L(1,0)\oplus L(0,1),\ p=3]$\\\hline
$(r,s)\in C_1$ & $(p-s-2,p-r-2)$ & $K$ & $[K\oplus L(1,0)\oplus L(0,1),\ p=3]$\\
& $(p-r-2,-p+r+s+1)$ & $L(1,0)$ & $[K\oplus L(1,0)\oplus L(0,1),\ p=3]$\\
& $(-p+r+s+1,p-s-2)$ & $L(0,1)$ & $[K\oplus L(1,0)\oplus L(0,1),\ p=3]$\\\hline
$(r,s),$ & $(p-1,r)$  & $L(1,0)$ &\\
$r+s=p-2$ & $(s,p-1)$  & $L(0,1)$ &\\\hline
$(r,p-1)$ & $(p-r-2,r)$  & $L(1,0)$ &\\\hline
$(p-1,s)$ & $(s,p-s-2)$  & $L(0,1)$ &\\\hline
\end{tabular}
\end{center}}
where in the above, $C_0$ is the fundamental alcove so that $C_0\cap X_+=\{(r,s)\in X_+:r+s<p-2\}$ and $C_1$ is the upper alcove in $X_1$ with $C_1\cap X_+=\{(r,s)\in X_+:r,s\leq p-2; r+s>p-2\}$.\end{prop}

We need the following lemma.

\begin{lemma}\label{strofind} Let $\lambda\in X_1$ be a restricted dominant weight which is $G_1$-linked to $(0,0)$. Then up to duality, the structure of $H^0(\lambda)$ is as follows:
\begin{center}\begin{tabular}{|l|l|}\hline
$\lambda$ & $H^0(\lambda)$\\\hline
$(0,0)$ & $K$\\
$(p-3,0)$ & $L(p-3,0),\ (p>3)$\\
$(p-2, 1)$ & $L(p-3,0)|L(p-2,1),\ (p>3)$\\
$(p-2,p-2)$ & $K|L(p-2,p-2), (p>2)$\\\hline\end{tabular}\end{center}
where $L(\lambda)|L(\mu)$ indicates a uniserial module of length $2$ with head $L(\lambda)$ and socle $L(\mu)$.\end{lemma}
\begin{proof}The first two items in the table follow since the weights concerned are in the lowest alcove, $C_0$ ($\bar C_0$ if $p=2$). The last two follow by an application of \cite[II.6.24]{Jan03}.\end{proof}



Using the graph automorphism of $G$, one sees that, $^\tau H^i(G_1,L(a,b))^*\cong H^i(G_1,L(b,a))$ for $\tau$ the antiautomorphism of \cite[II.1.16]{Jan03}, thus

\begin{prop}\label{cfk}Let $(a,b)\in X_1$. Then the non-zero values of $H^i(G_1,L(a,b))^{[-1]}$ for $i=0,1,2$ up to duals are given below by their structure as $G$-modules:

\begin{center}\begin{tabular}{|c c c c|}\hline
$i$ & $p$ & $(a,b)$ & $H^i(G_1,L(a,b))^{[-1]}$\\
\hline\hline
$0$ & any & $(0,0)$ & $K$\\
$1$ & $p>3$ & $(p-2,1)$ & $L(1,0)$\\
&&$(p-2,p-2)$ & $K$\\
& $p=3$ & $(1,1)$ & $K\oplus L(1,0)\oplus L(0,1)$\\
& $p=2$ & $(0,1)$ & $L(1,0)$\\
$2$&$p>3$ & $(0,0)$ & $L(1,1)$\\
&&$(p-3,0)$ & $L(1,0)$\\
& $p=3$ & $(0,0)$ & $H^0(1,1)\oplus L(1,0)\oplus L(0,1)$\\
&$p=2$ & $(0,0)$ & $L(1,1)$\\
&& $(1,0)$ & $H^0(2,0)\cong L(0,1)/L(2,0)$\\\hline\end{tabular}\end{center}
where when $p=3$, $H^0(1,1)\cong \mathfrak g^*$ is the uniserial module with head $L(0,0)$ and socle $L(1,1)$ by \ref{strofind}. 
\end{prop}
\begin{proof} By the linkage principle for $G_1$ we need only consider the weights which are $G_1$-linked to $(0,0)$ as listed in \ref{linkages}. 

For $i=0$, it is clear that $H^0(G_1,L(a,b))\cong L(a,b)^{G_1}\neq 0$ if and only if $(a,b)=(0,0)$. Also it is clear that  $H^0(G_1,L(0,0))=K$.

For $i=1$, one uses $\mu=0$ in \ref{yehia}. (Alternatively, use $\lambda=0$ and argue by duality.)

For $i=2$, we have from \ref{strofind} that $L(a,b)=H^0(a,b)$, except for $(a,b)=(p-2,p-2)$, so we may read the result off from \cite[Thm 6.2 (a),(b)]{BNP07} for $p\geq 3$ and $(a,b)\neq (p-2,p-2)$. 

We see from the Table 1 above that if $(p-2,p-2)=w\centerdot 0+p\lambda$  for some $w$, then  $w=w_0$, or $p=3$ and $w=w_0, s_\alpha$, or $s_\beta$. In any case $l(w)$ is odd, so we see from \cite[Thm 6.2(b)]{BNP07} that $H^2(G_1,H^0(p-2,p-2))=0$. Now, associated to the short exact sequence $0\to L(p-2,p-2)\to H^0(p-2,p-2)\to K\to 0$ of $G_1$-modules, there is a long exact sequence of $G$-modules of which part is \[H^1(G_1,K)\to H^2(G_1,L(p-2,p-2))\to H^2(G_1,H^0(p-2,p-2))\]
but since $H^1(G_1,K)=0$ by \cite[II.4.11]{Jan03}, this becomes
\[0\to H^2(G_1,L(p-2,p-2))\to 0\]
so we deduce that $H^2(G_1,L(p-2,p-2))=0$ as claimed.

Similarly, associated to the short exact sequence $0\to L(p-2,1)\to H^0(p-2,1)\to L(p-3,0)\to 0$ there is the exact sequence \[H^1(G_1,L(p-3,0))\to H^2(G_1,L(p-2,1))\to H^2(G_1,H^0(p-2,1))\] and the first and last terms are zero (by \ref{yehia} and by \cite[6.2(b)]{BNP07}, respectively). Thus $H^2(G_1,L(p-2,1))=0$ as claimed.

For $p=2$, the weights $G_1$-linked to $(0,0)$ are $(0,0)$, $(1,0)$ and $(0,1)$. Since for these weights, $H^0(\lambda)\cong L(\lambda)$, we get all $H^2(G_1,L(\lambda))$ from \cite[Thm 5.1.2]{Wri11}. (Note that the published version does not contain Appendix C. We have \cite[Theorem C.1.3]{Wri11} gives $H^2(B_1,(1,0))^{[-1]}\cong \nu_w$ where $(p-3,0)+2\nu_w=(1,0)$, i.e. $\nu_w=(2,0)$. Thus $H^2(G_1,H^0(1,0))^{[-1]}=\mathrm{Ind}_B^G(2,0)=H^2(2,0)$ as required.)
\end{proof}

\begin{corollary}\label{cor}Provided $p>2$, if $\lambda\in X_1$ and $H^1(G_1,L(\lambda))\neq 0$ then $H^0(G_1,L(\lambda))=H^2(G_1,L(\lambda))=0$.\end{corollary}

The third ingredient is the values of $\Ext^1_G(L(\lambda),L(\mu))$ from the main result of \cite[4.2.3]{Yeh82}. We will need the values of $\Ext^1_G(L(\lambda),L(\mu))$ only for $\lambda=(0,0),\ (1,0),\ (0,1),$ and $(1,1)$.

\begin{lemma}\label{ext1}If $\Ext^1_G(L(\lambda),L(\mu))$ is non-zero then $\Ext^1_G(L(\lambda),L(\mu))=K$. Let $\lambda=(0,0),\ (1,0),\ (0,1),$ or $(1,1)$. We list the values of $L(\mu)$ in the tables below affording a non-zero value of $\Ext^1_G(L(\lambda),L(\mu))$.

for $p>3$
{\small
\begin{center}\begin{tabular}{|c|l|}\hline 
$\lambda$ & $L(\mu)$\\
\hline\hline $(0,0)$ & $L(p-2,p-2)^{[i]},\ {L(1,p-2)^{[i]}\otimes L(1,0)^{[i+1]}},\ {L(p-2,1)^{[i]}\otimes L(0,1)^{[i+1]}}$\\
\hline
$(1,0)$ & $L(p-2,p-3),\ L(p-3,2)\otimes L(0,1)^{[1]},\ L(2,p-2)\otimes L(1,0)^{[1]}$,\\
& $L(1,0)\otimes L(p-2,p-2)^{[i+1]},\ {L(1,0)\otimes L(1,p-2)^{[i+1]}\otimes L(1,0)^{[i+2]}},$\\& ${L(1,0)\otimes L(p-2,1)^{[i+1]}\otimes L(0,1)^{[i+2]}}$\\
\hline
$(0,1)$ & $L(p-3,p-2),\ L(2,p-3)\otimes L(1,0)^{[1]},\ L(p-2,2)\otimes L(0,1)^{[1]}$,\\
& $L(0,1)\otimes L(p-2,p-2)^{[i+1]},\ {L(0,1)\otimes L(p-2,1)^{[i+1]}\otimes L(0,1)^{[i+2]}},$\\& ${L(0,1)\otimes L(1,p-2)^{[i+1]}\otimes L(1,0)^{[i+2]}}$\\
\hline
$(1,1)$ & $L(p-3,p-3),\ L(3,p-3)\otimes L(1,0)^{[1]},\ L(p-3,3)\otimes L(0,1)^{[1]}$,\\
& $L(1,1)\otimes L(p-2,p-2)^{[i+1]},\ {L(1,1)\otimes L(p-2,1)^{[i+1]}\otimes L(0,1)^{[i+2]}},$\\& ${L(1,1)\otimes L(1,p-2)^{[i+1]}\otimes L(1,0)^{[i+2]}}$\\\hline\end{tabular}\end{center}}

for $p=3$
{\small \begin{center}\begin{tabular}{|c|l|}\hline $\lambda$ & $L(\mu)$\\
\hline\hline $(0,0)$ & $L(1,1)^{[i]}$, $L(1,1)^{[i]}\otimes L(1,0)^{[i+1]}$,
$L(1,1)^{[i]}\otimes L(0,1)^{[i+1]}$\\
\hline
$(1,0)$ & $L(2,1)\otimes L(1,0)^{[1]}$,
$L(0,2)\otimes L(0,1)^{[1]}$, $L(1,0)\otimes L(1,1)^{[i+1]}$,\\
&$L(1,0)\otimes L(1,1)^{[i+1]}\otimes L(1,0)^{[i+2]}$, $L(1,0)\otimes L(1,1)^{[i+1]}\otimes L(0,1)^{[i+2]}$\\
\hline
$(0,1)$ &  $L(1,2)\otimes L(0,1)^{[1]}$, 
$L(2,0)\otimes L(1,0)^{[1]}$, 
$L(0,1)\otimes L(1,1)^{[i+1]}$,\\
& $L(0,1)\otimes L(1,1)^{[i+1]}\otimes L(1,0)^{[i+2]}$, 
$L(0,1)\otimes L(1,1)^{[i+1]}\otimes L(0,1)^{[i+2]}$ \\
\hline
$(1,1)$ & $L(0,0)$, $L(1,1)\otimes L(1,0)^{[1]}$, $L(1,1)\otimes L(0,1)^{[1]}$, $L(1,1)\otimes L(1,1)^{[i+1]}$, \\
&$L(1,1)\otimes L(1,0)^{[i+1]}\otimes L(0,1)^{[i+2]}$, $L(1,1)\otimes L(0,1)^{[i+1]}\otimes L(1,0)^{[i+2]}$ \\\hline\end{tabular}\end{center}}

for $p=2$
{\small \begin{center}\begin{tabular}{|c|l|}\hline $\lambda$ & $L(\mu)$\\
\hline\hline $(0,0)$ & ${L(1,0)^{[i]}\otimes L(1,0)^{[i+1]}},\ {L(0,1)^{[i]}\otimes L(0,1)^{[i+1]}}$\\
\hline
$(1,0)$ & $L(1,0)\otimes L(0,1)^{[1]}$, $L(1,0)\otimes L(1,0)^{[i+1]}\otimes L(1,0)^{[i+2]}$,\\& ${L(1,0)\otimes L(0,1)^{[i+1]}\otimes L(0,1)^{[i+2]}}$\\
\hline
$(0,1)$ & $L(0,1)\otimes L(1,0)^{[1]}$, $L(0,1)\otimes L(0,1)^{[i+1]}\otimes L(0,1)^{[i+2]}$,\\& ${L(0,1)\otimes L(1,0)^{[i+1]}\otimes L(1,0)^{[i+2]}}$\\
\hline
$(1,1)$ & ${L(1,1)\otimes L(1,0)^{[i+1]}\otimes L(1,0)^{[i+2]}}$, $L(1,1)\otimes L(0,1)^{[i+1]}\otimes L(0,1)^{[i+2]}$ \\\hline\end{tabular}\end{center}
} where $i\geq 0$.
\end{lemma}

\begin{prop}\label{23page}In the spectral sequence of \ref{spec} applied to $V=L(a,b)$, 
\begin{enumerate}
\item $E_2^{20}=E_\infty^{20}$
\item $E_2^{02}=E_\infty^{02}$
\item $E_2^{11}=E_\infty^{11}$
\end{enumerate}
\end{prop}
\begin{proof}
First suppose that $p>2$.

Since $E_3^{nm}$ is the cohomology of \[E_2^{n-2,m+1}\to E_2^{nm}\to E_2^{n+2,m-1}\] we have that $E_2^{nm}=E_3^{nm}$ provided the following statement holds: \begin{equation}\label{star}E_2^{n-2,m+1}=E_2^{n+2,m-1}=0\text{ whenever }E_2^{nm}\neq 0.\end{equation} We now show that (\ref{star}) holds for all values of $\lambda=(a,b)$ and all $(n,m)$ with $n+m=2$.

If $E_2^{nm}\neq 0$ then $H^m(G_1,V)\neq 0$ so for $p>2$, (\ref{star}) follows using \ref{cor} in the cases $(n,m)=(0,2),(1,1)$, or $(2,0)$. Since $E_3^{20}=E_\infty^{20}$ and $E_3^{11}=E_\infty^{11}$ for any first quadrant spectral sequence, (i) and (ii) above are true. It remains to check that $E_3^{02}=E_4^{02}$. Now $E_4^{02}$ is the cohomology of $E_3^{-3,2}\to E_3^{02}\to E_3^{30}$ so we are done provided we can show $E_2^{30}=0$ whenever $E_2^{02}\neq 0$. 

Suppose $E_2^{02}\neq 0$ and $E_2^{30}\neq 0$. Then the latter implies $\lambda_0=(0,0)$ by \ref{cfk}. In that case, when $p\neq 3$, $E_2^{02}=H^0(G,L(1,1)\otimes L(\lambda'))=\Hom_G(L(1,1),L(\lambda'))$. So $\lambda'=(1,1)$ and $L(\lambda)=L(1,1)^{[1]}$. But then $E_2^{30}=H^3(G,H^0(G_1,K)^{[-1]}\otimes \lambda)=H^3(G,L(1,1))=0$ since for $p\neq 3$, $(1,1)$ is not linked to $(0,0)$.

For $p=3$ we get similarly that $L(\lambda)=L(1,1)^{[1]}$, $L(0,1)^{[1]}$ or $L(1,0)^{[1]}$. Only $(1,1)$ is linked to $0$ for $p=3$, but $E_2^{30}=H^3(G,L(1,1))=H^2(G,H^0(1,1)/L(1,1))$ by \cite[II.4.14]{Jan03} and so by \ref{strofind}, this is $H^2(G,K)=0$.

Lastly we deal with the case $p=2$. 

We cannot have $E_2^{01}\neq 0$ and $E_2^{20}\neq 0$ by \ref{cfk}. So (i) holds by (\ref{star}). Similarly, we cannot have $E_2^{11}\neq 0$ and $E_2^{30}\neq 0$ so (iii) holds again by (\ref{star}).

Now suppose $E_2^{02}\neq 0$ and $E_2^{21}\neq 0$. From \ref{cfk} we must deal with the two possibilities $\lambda_0=(1,0)$ or $(0,1)$. By duality of the following argument, we may assume the first. Since $E_2^{02}=H^0(G,H^2(G_1,L(1,0))^{[-1]}\otimes \lambda' )\neq 0$ we get $H^0(G,H^0(2,0)\otimes L(\lambda'))=\Hom_G(V(0,2),L(\lambda'))\neq 0$ and thus $\lambda'=(0,2)$ giving $L(\lambda)=L(1,0)\otimes L(0,1)^{[2]}$. Putting this into $E_2^{21}$ we get
\begin{align*}E_2^{21}&=H^2(G,H^1(G_1,L(1,0))^{[-1]}\otimes L(0,1)^{[1]})\\
&=H^2(G,L(0,1)\otimes L(0,1)^{[1]})\end{align*}

Continuing to analyse the latter group using the LHS spectral sequence with terms ${E'}_2^{ij}$, we have ${E'}_2^{20}=0$ 
\begin{align*}{E'}_2^{02}&=H^0(G,H^2(G_1,L(0,1))^{-1}\otimes L(0,1))\cong \Hom_G(V(2,0),L(0,1))=0,\\
{E'}_2^{11}&=H^1(G,H^1(G_1,L(0,1)^{[-1]})\otimes L(0,1))\cong \Ext^1_G(L(0,1),L(0,1))=0;\end{align*}
where the latter equality holds since $G$-modules have no self-extensions. It follows that $E_2^{21}=0$.

Thus we have shown that $E_2^{02}=E_3^{02}$. We must now check that $E_3^{02}=E_4^{02}$ as above. The argument used above for $p\neq 3$ applies in this case showing that $E_2^{30}=0$ so that indeed $E_3^{02}=E_4^{02}=E_\infty^{02}$ and (ii) holds.
\end{proof}

Using (\ref{sumup}) and \ref{23page}, we get

\begin{corollary}\label{h2} For $V=L(a,b)$, $H^2(G,V)=E_2^{02}\oplus E_2^{11}\oplus E_2^{20}$.\end{corollary}

We can now finish the

{\it Proof of Theorem 1:} 

If $V$ is the trivial module then $H^2(G,V)=0$ by, for example, dimension shifting. So let $V=L(\lambda)=L(\mu)^{[d]}$ with $\mu_0\neq (0,0)$.

We have from \ref{h2} that $H^2(G,V)=E_2^{20}\oplus E_2^{11}\oplus E_2^{02}$. Assume $H^2(G,V)\neq 0$.

Let $p>2$ initially. 

Suppose $E_2^{11}=H^1(G,H^1(G_1,L(\lambda_0))^{[-1]}\otimes L(\lambda'))\neq 0$. Then by \ref{cor}, $E_2^{02}=E_2^{20}=0$. Additionally, $H^1(G_1,L(\lambda_0))\neq 0$ so $\lambda_0= (p-2,1),\ (1,p-2)$ or $(p-2,p-2)$ by \ref{cfk}.  In particular $d=0$. Now
\[E_2^{11}=H^1(G,H^1(G_1,L(\lambda_0))^{[-1]}\otimes L(\lambda'))=\Ext_G^1({H^1(G_1,L(\lambda_0))^{[-1]}}^*,L(\lambda'))\]
and we can read off the values of the latter from \ref{ext1} in conjuction with \ref{cfk}. For example, if $p=3$ then $\lambda_0=(1,1)$ and so $H^1(G_1,L(\lambda_0))^{[-1]}=L(0,0)\oplus L(1,0)\oplus L(0,1)$ by \ref{cfk}.  Then $H^2(G,V)=E_2^{11}=\Ext_G^1(L(0,0),L(\lambda'))\oplus\Ext_G^1(L(0,1),L(\lambda'))\oplus\Ext_G^1(L(1,0),L(\lambda'))$. Since $E_2^{11}\neq 0$ by assumption, we see from \ref{ext1} that
\begin{align*}L(&\lambda')\in\{L(1,1)^{[i]},\ L(1,1)^{[i]}\otimes L(1,0)^{[i+1]},\ L(0,2)\otimes L(0,1)^{[1]}, \\& L(2,1)\otimes L(1,0)^{[1]},\ L(1,0)\otimes L(1,1)^{[i+1]},\\& L(1,0)\otimes L(1,1)^{[i+1]}\otimes L(1,0)^{[i+2]},\ L(1,0)\otimes L(1,1)^{[i+1]}\otimes L(0,1)^{[i+2]}\}\end{align*} up to duals, for any $i\geq 0$. For each value of $\lambda'$ above, it follows that $H^2(G,V)=K$ since precisely one of the three terms in $E_2^{11}$ is $K$. One then observes that each $L(\lambda)=L(\lambda_0)\otimes L(\lambda')^{[1]}$ with $\lambda_0=(1,1)$ appears in the statement of Theorem 1. The other cases (for $p>3$) are similar and easier.

Now we assume that $E_2^{11}=0$ and so under the standing assumption that $H^2(G,V)\neq 0$ we must have one or both of $E_2^{02}$ and $E_2^{20}\neq 0$. Suppose it is the former which is non-zero; we will deduce the possible values of $\lambda$ and then show that this implies that $E_2^{20}=0$. Now, since $E_2^{02}\neq 0$, we have $H^2(G_1,L(\lambda_0))\neq 0$ and so up to duals, $\lambda_0=(0,0)$ or $(p-3,0)$ from \ref{cfk} and we also read off the value of  $H^2(G_1,L(\lambda_0))$ from \ref{cfk}. As \[E_2^{02}=\Hom_G({H^2(G_1,L(\lambda_0))^{[-1]}}^*,L(\lambda'))\]
we have that $L(\lambda')$ is an irreducible quotient of ${H^2(G_1,L(\lambda_0))^{[-1]}}^*$. If $p=3$, then $\lambda_0=(0,0)$ and we get ${H^2(G_1,L(\lambda_0))^{[-1]}}^*=V(1,1)\oplus L(1,0)\oplus L(0,1)$ by \ref{cfk} giving $\lambda'=(1,1),\ (1,0)$, or $(0,1)$, and so $L(\lambda)=L(1,1)^{[1]}, L(1,0)^{[1]}$ or  $L(0,1)^{[1]}$ respectively, with each affording $E_2^{02}=K$. Similarly, if $p>3$, we get $L(\lambda)=L(1,1)^{[1]}$, or  $L(p-3,0)\otimes L(1,0)^{[1]}$. In all cases it is easy to see that $E_2^{20}=0$: $H^0(G_1,L(\lambda_0))=0$ unless $\lambda_0=(0,0)$; in this case, $E_2^{20}=H^2(G,L(\lambda'))=0$ for all instances of $\lambda'$. For example, if $p=3$, and $L(\lambda)=L(1,1)^{[1]}$ then $E_2^{20}=H^2(G,L(1,1))=H^1(G,L(0,0))=0$ as argued earlier. Thus when $E_2^{02}\neq 0$, $H^2(G,V)=E_2^{02}=K$.

Lastly if $E_2^{02}=E_2^{11}=0$ then $H^2(G,V)=E_2^{20}=H^2(G,H^0(G_1,L(\lambda_0))^{[-1]}\otimes L(\lambda'))$ and so $\lambda_0=(0,0)$ and $H^2(G,L(\lambda))\cong H^2(G,L(\lambda)^{[-1]})$ as $L(\lambda')=L(\lambda)^{[-1]}$ and so $L(\lambda)$ is a Frobenius twist of one of the modules already discovered above and $H^2(G,V)=K$.

Now assume that $p=2$. If $E_2^{11}\neq 0$, then $H^1(G_1,L(\lambda_0))\neq 0$ and so $\lambda_0=(1,0)$ (up to duality) with $H^1(G_1,L(\lambda_0))^{[-1]}=L(0,1)$. Then $E_2^{11}=\Ext_G^1(L(1,0),L(\lambda'))$ and so \begin{align*}L(\lambda)'\in \{L(1,0)\otimes L(0,1)^{[1]},\ L(1,0)\otimes L(1,0)^{[i+1]}\otimes L(1,0)^{[i+2]},\\ L(1,0)\otimes L(0,1)^{[i+1]}\otimes L(0,1)^{[i+2]}\}\end{align*} by \ref{ext1}. Observe that $E_2^{20}=0$ as $H^0(G_1,L(\lambda_0))=0$. We examine \[E_2^{02}=H^0(G,H^2(G_1,L(\lambda_0))^{[-1]}\otimes L(\lambda')).\] When $L(\lambda_0)=(1,0)$ we have $E_2^{02}=\Hom_G(V(0,2),L(\lambda')^*)$ by \ref{cfk} and so $\lambda'=(0,2)$ for this term to be non-zero, yet this is not included in the possibilities for $\lambda'$ above. Thus if $E_2^{11}\neq 0$, $E_2^{20}=E_2^{02}=0$ and $H^2(G,V)=E_2^{11}=K$.

If $E_2^{02}\neq 0$ similarly we get that $L(\lambda)=L(1,1)^{[1]}$, $L(0,1)\otimes L(1,0)^{[2]}$ or $L(1,0)\otimes L(0,1)^{[2]}$. In each of these cases $E_2^{11}=E_2^{20}=0$ with $H^2(G,V)=E_2^{02}=K$.

Lastly if $E_2^{02}=E_2^{11}=0$ then $H^2(G,V)=E_2^{20}=H^2(G,H^0(G_1,L(\lambda_0))^{[-1]}\otimes L(\lambda'))$ and so $\lambda_0=0$ and $H^2(G,L(\lambda))\cong H^2(G,L(\lambda)^{[-1]})$. Thus $L(\lambda)$ is a Frobenius twist of one of the modules discovered above and $H^2(G,V)=K$.

\section{Finite groups of Lie type: Proof of Theorem 2}

In this section we use results of Bendel, Nakano and Pillen to get information about the cohomology for the Chevalley groups $SL(3,q)$ with $q=p^r$. Our Theorem 2 is a generalisation of the following theorem to second degree cohomology in the case $G=SL_3$---we give just the split version. Here $X_r$ denotes weights for $G$ which are $p^r$-restricted. Let $h$ be the Coxeter number of $G$.

\begin{theorem}[cf. {\cite[5.5]{BNP06}}] \label{bnp06}Let $G$ be a split reductive algebraic group defined over an algebraically closed field $K$ with char $K=p$. Assume $r\geq 2$ and let $s=\left[\frac{r}{2}\right]$. Assume also $p^{s-1}(p-1)\geq h$. Given $\lambda\in X_r$, let $\lambda=\lambda_0+p^s\lambda_1$ with $\lambda_0\in X_s$ and $\lambda_1\in X_{r-s}$. Define $\tilde\lambda=\lambda_1+p^{r-s}\lambda_0$. Then \[H^1(G(q),L(\lambda))\cong\begin{cases}H^1(G,L(\tilde\lambda))\text{ if } \lambda_0\in\Gamma_h-\{0\}\\
H^1(G,L(\lambda))\text{ else.}\end{cases}\]\end{theorem}
Here $\Gamma_h$ is the set of weights $\{\nu\in X_+ | \langle\nu,\alpha_0^\vee\rangle<h\}$, where $X_+$ denotes the set of dominant weights, $h$ is the Coxeter number for the root system of $G$ and $\alpha_0^\vee$ is the coroot associated to the highest root.

The starting point for the proof of Theorem 2 is the theorem below. Define $X[s]=\{\nu\in X_+ | \langle \nu,\alpha_0^\vee\rangle < 2s(h-1)\}$.

\begin{theorem}[See the proof of {\cite[7.4]{BNP01}}]
\label{BNP01}Let $s\geq 1$, $p\geq (2s+1)(h-1)$ and let $\lambda\in X_r$ with $0\leq i\leq s$.
\begin{enumerate} \item If $\lambda$ is $p$-regular, then
\[H^i(G(\mathbb F_q),L(\lambda))\cong \bigoplus_{\nu\in X[s]}\Ext^i_G(L(\nu),L(\lambda+p^r\nu)).\]
\item If $\lambda$ is $p$-singular, then 
\[H^i(G(\mathbb F_q),L(\lambda))=0.\]\end{enumerate}
\end{theorem}

In the special case that $G=SL_3$ (so $h=3$) and $s=2$, the conclusions of the theorem are valid when $p\geq 11$. Since $\alpha_0=\alpha_1+\alpha_2$, one checks that $\nu=(a,b)\in X[2]$ implies that $a+b<8$. In particular, as $p\geq 11$, $\nu$ is restricted; indeed $\nu\in C_0$, the fundamental alcove.

For Theorem 2, we will need a slight generalisation of the `$\tilde\lambda$' notation above. The following notion will be very familiar to mathematicians working with the cohomology of finite groups of Lie type in defining characteristic, but we are unaware of it ever having appeared in the literature.

\begin{defn}Let $G$ be a split, reductive algebraic group defined over an algebraically closed field $K$ of characteristic $p$. Set $q=p^r$ for some $r\in\mathbb N$ and let $L$ be a $q$-restricted simple module for $G$, i.e. the highest weight of $L$ is in $X_r$. Let $F$ be a standard Frobenius map, so that $G(q)=G(K)^{F^r}$. Then for any $s\geq 0$ the Frobenius twist $L^{[s]}\downarrow G(q)$ is again a simple $G(q)$-module and thus is also a simple $G$-module, $M$ say. We say that {\it $M$ is obtained from $L$ by $q$-wrapping}; more precisely we say that {\it $M$ is the $s$-fold $q$-wrap of $L$}, and write $M=L^{\{s\}}$.\end{defn}

\begin{remark}Note by Steinberg's tensor product theorem, if $\lambda=\lambda_0+p\lambda_1+\dots+p^{r-1}\lambda_{r-1}$, with each $\lambda_i\in X_1$, we have
\[L(\lambda)=L(\lambda_0)\otimes L(\lambda_1)^{[1]}\otimes \dots \otimes L(\lambda_{r-1})^{[r-1]}.\]
Thus, if $1\leq s< r$, it is easy to see that the $s$-fold $q$-wrap of this is $M=$
\[L(\lambda_{r-s})\otimes L(\lambda_{r-s+1})^{[1]}\otimes \dots\otimes L(\lambda_{r-1})^{[s-1]}\otimes L(\lambda_1)^{[s]}\otimes\dots\otimes L(\lambda_{r-s-1})^{[r-1]}.\]
If $s=0$, clearly $M=L$ and for $s\geq r$, note that $L^{\{s\}}\cong L^{\{s-r\}}$.

In addition, if one writes $\lambda=\lambda'+p^{r-s}\lambda''$ for $\lambda'\in X_{r-s}$ and $\lambda''\in X_s$, then let $\tilde\lambda=\lambda''+p^s\lambda'$. Then $M=L(\tilde\lambda)$, which connects with the notation in \ref{bnp06}.
\end{remark}
\begin{T2}Let $\lambda\in X_r$  be a $q$-restricted dominant weight, with $q=p^r$ and suppose $H^2(G(q),L(\lambda))\neq 0$. Then $H^2(G(q),L(\lambda))=K$ and precisely one of the following holds:
\begin{enumerate}\item $H^2(G(q),L(\lambda))\cong H^2(G,L(\mu))\cong K$ where $L(\lambda)$ is $q$-wrapped from some $L(\mu)$, such that $\mu$ is q-restricted  (and is thus determined by Theorem 1).
\item $r=2$ and $L(\lambda)$ or its dual is $L(2,p-2)\otimes L(2,p-2)^{[1]}$ or $L(p-3,2)\otimes L(2,p-3)^{[1]}$.
\item $r=1$ and $L(\lambda)$ is $L(1,1)$, $L(0,p-4)$ or $L(p-4,0)$.
\end{enumerate}
\end{T2}
\begin{proof}By \ref{BNP01} we may assume that $\lambda$ is $p$-regular and we have that \begin{equation}\label{zero}H^2(G(\mathbb F_q),L(\lambda))\cong \bigoplus_{\nu\in X[2]}\Ext^2_G(L(\nu),L(\lambda+p^r\nu)).\end{equation}
We calculate the right hand side of this using translation functors. Let $T_\mu^\xi$ denote the translation functor corresponding to the weights $\mu$ and $\xi$; see \cite[II.7.6(1)]{Jan03}. In the following, recall that $\nu\in X[2]\subseteq C_0$, the fundamental alcove; also, since $p>h$, that the weight $0\in C_0$.

Assuming $\Ext^2_G(L(\nu),L(\lambda+p^r\nu))\neq 0$, we can write $L(\lambda+p^r\nu)=L(w.\nu)$ for $w\in W_p$ by the linkage principle. Thus 
\begin{align*}\Ext^2_G(L(\nu),L(\lambda+p^r\nu))&\cong \Ext^2_G(T_0^\nu K,L(w.\nu))\tag{by \cite[II.7.15]{Jan03}}\\
&\cong \Ext^2_G(K,T_\nu^0 L(w.\nu))\tag{by \cite[II.7.6(2)]{Jan03}}\\
&\cong \Ext^2_G(K,L(w.0))\tag{by \cite[II.7.15]{Jan03} again}\\
&\cong H^2(G,L(w.0)).\end{align*}

In particular, by Theorem 1, \begin{equation}\Ext^2_G(L(\nu),L(\lambda+p^r\nu))\text{ is either $0$- or $1$-dimensional.}\label{one}\end{equation}

Thus $L(w.0)$ is determined by Theorem 1, subject to the condition that $T_0^\nu L(w.0)=L(w.\nu)=L(\lambda)\otimes L(\nu)^{[r]}$. These $L(w.0)$ are easy to find---as $0$ and $\nu$ are in the interior of the same alcoves, we have $T_0^\nu(L(\lambda_0)\otimes L(\lambda')^{[1]})=T_0^\nu(L(\lambda_0))\otimes L(\lambda')^{[1]}$ for any $\lambda=\lambda_0+p\lambda'$ and we can look at the top tensor factors of the modules in Theorem 1 for twists of valid $L(\nu)$ with $\nu\in X[2]$. For a given value of $\nu$, fix such an element $w$ and note that from the linkage principle, we have $\Ext^2_G(L(\nu'),L(\lambda)\otimes L(\nu)^{[r]})=\Ext^2_G(L(\nu'),L(w.\nu))=0$ for any $\nu'\neq \nu$ with $\nu'\in X[2]$. Thus we have that 
\begin{equation}\Ext^2_G(L(\nu),L(\lambda+p^r\nu))\neq 0\text{ for at most one value of $\nu$.}\label{two}\end{equation}

By statements (\ref{zero}), (\ref{one}) and (\ref{two}) we see that $H^2(G(q),L(\lambda))$ is at most one-dimensional. This proves the first part of the statement of the theorem. It remains to establish the non-zero values of $L(\lambda)$. If $\nu=0$, we use Theorem 1 to identify $q$-restricted weights $\lambda$ with $H^2(G,L(\lambda))\neq 0$. Then $H^2(G(q),L(\lambda))=K$ also, and these are included in part (i) of the theorem since $L(\lambda)=L(\lambda)^{\{0\}}$.

Scanning the statement of Theorem 1 for modules $L(w.0)$ satisfying $T_0^\nu L(w.0)=L(w.\nu)=L(\lambda)\otimes L(\nu)^{[r]}$, we find that one of the following hold, up to duality:
\begin{enumerate}\item $\nu=(1,1)$, $r\geq 1$ and $L(w.0)=L(1,1)^{[r]}$,
\item $\nu=(1,0)$ with \begin {enumerate}\item $r\geq 1$ and $L(w.0)=L(0,p-3)^{[r-1]}\otimes L(1,0)^{[r]}$,
\item  $r\geq 2$ and $L(w.0)=L(p-2,1)^{[r-2]}\otimes L(2,p-3)^{[r-1]}\otimes L(1,0)^{[r]}$,
\item  $r\geq 2$ and $L(w.0)=L(1,p-2)^{[r-2]}\otimes L(2,p-2)^{[r-1]}\otimes L(1,0)^{[r]}$,
\item  $r\geq 2$ and $L(w.0)=L(p-2,p-2)^{[s]}\otimes L(1,p-2)^{[r-1]}\otimes L(1,0)^{[r]}$, for any $0\leq s\leq r-2$,
\item  $r\geq 3$ and $L(w.0)=L(1,p-2)^{[s]}\otimes L(1,0)^{[s+1]}\otimes L(1,p-2)^{[r-1]}\otimes L(1,0)^{[r]}$, for any $0\leq s\leq r-3$,
\item  $r\geq 3$ and $L(w.0)=L(p-2,1)^{[s]}\otimes L(0,1)^{[s+1]}\otimes L(1,p-2)^{[r-1]}\otimes L(1,0)^{[r]}$ for any $0\leq s\leq r-3$.
\end{enumerate}\end{enumerate}

We work through the possibilities in turn, again up to duality. 

Firstly, assume the restricted part of $L(w.0)$ is zero; i.e. we are in case (i), case (ii)(a) if $r\geq 2$, cases (ii)(b) or (c) if $r\geq 3$, or the remaining cases if $s>0$. Then $L(w.\nu)=T_0^\nu L(w.0)=L(\nu)\otimes L(w.0)$. Unless $r=1$ and we are in case (i), one can see that $L(\lambda)$ is a $q$-wrap of some $L(\mu)$ with $H^2(G,L(\mu))\neq 0$. For example if $r\geq 3$ and $L(w.0)=L(p-2,1)^{[r-2]}\otimes L(2,p-2)^{[r-1]}\otimes L(1,0)^{[r]}$ (case (ii)(b)), then $L(w.\nu)=L(1,0)\otimes L(p-2,1)^{[r-2]}\otimes L(2,p-2)^{[r-1]}\otimes L(1,0)^{[r]}$, and so $L(\lambda)=L(1,0)\otimes L(p-2,1)^{[r-2]}\otimes L(2,p-2)^{[r-1]}$. But if $L(\mu)=L(p-2,1)^{[r-3]}\otimes L(2,p-2)^{[r-2]}\otimes L(1,0)^{[r-1]}$ then $H^2(G,L(\mu))\cong K$ and $L(\lambda)=L(\mu)^{\{1\}}$. (Indeed, $L(\lambda)=(L(w.0)^{[-1]})^{\{1\}}$ for these cases.) If $r=1$ and $L(w.0)=L(1,1)^{[1]}$, then $L(w.\nu)=L(1,1)\otimes L(1,1)^{[1]}$ and $L(\lambda)=L(1,1)$. Thus the $2$-cohomology for $G(p)$ afforded by the module $L(1,1)$ does not arise from $p$-wrapping from any $p$-restricted module ($p$-wrapping, of course, does nothing).

Thus we may assume that the restricted part of $L(w.0)$ is non-zero, i.e. we are in one of the set of cases (ii) and we have equality in the condition on $r$, or $s=0$.

Let $r=1$. Then $L(w.0)=L(0,p-3)$ and $L(\lambda)=T_0^\nu L(0,p-3)=L(0,p-4)$.

Let $r=2$ and $s\neq 0$, then we find $L(\lambda)$ is one of $L(p-3,2)\otimes L(2,p-3)^{[1]}$, $L(2,p-2)\otimes L(2,p-2)^{[1]}$.

We are left with the case $s=0$. Then one of the following holds: $L(\lambda)=L(p-2,p-3)\otimes L(1,p-2)^{[r-1]}$; or $r\geq 3$ and $L(\lambda)$ is either $L(2,p-2)\otimes L(1,0)^{[1]}\otimes L(1,p-2)^{[r-1]}$ or $L(p-3,2)\otimes L(0,1)^{[1]}\otimes L(1,p-2)^{[r-1]}$. But for all of these we see $L(\lambda)$ is a $q$-wrap of one of the modules in Theorem 1: respectively, {\small
\begin{align*}L(p-2,p-3)\otimes L(1,p-2)^{[r-1]}&=(L(1,p-2)^{[r-2]}\otimes L(p-2,p-3)^{[r-1]})^{\{1\}},\\
 L(2,p-2)\otimes L(1,0)^{[1]}&\otimes L(1,p-2)^{[r-1]}\\=(L(1,&p-2)^{[r-3]}\otimes L(2,p-2)^{[r-2]}\otimes L(1,0)^{[r-1]})^{\{2\}}\text{ and }\\
 L(p-3,2)\otimes L(0,1)^{[1]}&\otimes L(1,p-2)^{[r-1]}\\=(L(1,&p-2)^{[r-3]}\otimes L(p-3,2)^{[r-2]}\otimes L(0,1)^{[r-1]})^{\{2\}}.\end{align*}}
The proof is complete.
\end{proof}
\begin{remark} The result \cite[Proposition 4.4]{Sah77} contains completely general information about second cohomology groups for $SL(n,q)$ with coefficients in its natural module $V$. In the case $n=3$, it says $H^2(G(q),V)=0$ unless $q=2$, $3^d$ for $d>1$, and $5$; in the exceptional cases, $H^2(G(q),V)=K$. We note this proposition agrees with Theorem 2 on their intersection. Note also that when $p=2$ and $r>1$, $H^2(G,V^{[r]})=K$; so one might have predicted the statement for $q=3^d$, though there would be little reason from Theorem 1 to suspect the exceptional cases when $q=2$ or $q=5$.\end{remark}

The way we have defined $q$-wrapping, none of the modules giving rise to non-trivial $2$-cohomology for $G(p)$ are ($p$-wraps of) modules affording non-trivial $2$-cohomology for $G$. In other words,
\begin{corollary}Let $p\geq 11$ and let $\lambda$ be a restricted weight. 

If $H^2(SL(3,p),L(\lambda))$ is non-trivial then it is one-dimensional and $\lambda$ is the weight $(1,1)$, $(p-4,0)$ or $(0,p-4)$.\end{corollary}
\begin{remark} The exceptional cases (ii) and (iii) in the statement of Theorem 2 shows that one will need, generically, some restriction on the power $r$ of $p$ taken for $G(q)$ to get a statement like that of \ref{bnp06} for second cohomology. That is to say, our result shows that even for the case $G=SL_3$, for each prime $p\geq 11$, there are four simple modules for $G(p^2)$ and three for $G(p)$ affording non-trivial $2$-cohomology which do not arise from $q$-wrapping. Perhaps $r\geq 3$ would suffice?\end{remark}

It would be nice to have more general results on how cohomology of simple modules for $G(q)$ arises from the $q$-wrapping of simple modules for $G$ whose cohomology is non-trivial. Specifically, we ask this

\begin{question}Do there exist constants $\xi=\xi(m)$ and constants $\zeta=\zeta(m,\Phi)$ with the property that if $G$ is a split, simple algebraic group with root system $\Phi$ over an algebraically closed field $K$ of characteristic $p\geq\zeta$ and $G(q)\leq G$ is a finite Chevalley group over the field $\mathbb F_q$ with $q=p^r$ such that $r\geq\xi$, then for any pair of $q$-restricted simple $G$-modules, $L(\lambda)$ and $L(\mu)$ we have $\Ext^m_{G(p^r)}(L(\lambda),L(\mu))\cong \Ext^m_G(L(\lambda)^{\{s\}},L(\mu)^{\{s\}})$ for some $0\leq s<r$?\end{question}

Theorem \cite[Theorem 5.6]{BNP06} shows that one can take $\zeta(1,\Phi)=h$ and $\xi(1)=3$.

\begin{remark}The way we have defined it, low values of $r$ compared to the degree of the cohomology are likely to cause problems: without going into the details, we are aware of many examples of situations where there are modules $L(\mu)$ with $\dim H^m(G,L(\mu)^{[s]})\gneqq\dim H^m(G,L(\mu)^{[s-t]})$ for various $t<s$. But we may well have $\dim H^m(G(p^v),L(\mu))=\dim H^m(G,L(\mu)^{[s]})$ for $v<s$. In other words, to see all cohomology for $G(p^v)$ as arising from cohomology for $G$, one might like to consider a notion of $q$-wrapping for non-restricted weights; for instance, one could say a simple $G(q)$-module is `$q$-wrapped' from a simple $G$-module if it is $q$-wrapped (as above) from some $q$-restricted untwist. We would urge caution, though. The example we saw of this for $G=SL_3$ is when $L(\mu)=L(1,1)$ and we have $H^2(G,L(\mu)^{[1]})=K$, $H^2(G,L(\mu))=0$ for all $p$ (by Theorem 1) and $H^2(G(p),L(\mu))=K$ for $p>7$ (by the above corollary). But when $p=2$, $L(\mu)=\text{St}$ and $H^2(G(2),\text{St})=0$ since the Steinberg module is projective for $G(2)$. For low $p$, one would generally expect to find many such weights $\mu$ becoming singular in this fashion.\end{remark}

Somehow the reason one needs a restriction on $r$ to ensure cohomology arises from $q$-wrapping seems to be connected with the number of non-trivial tensor factors in a Steinberg decomposition of the simple module. For $r=2$ we turned up an exception arising from the $G$-module $L(p-2,1)\otimes L(2,p-3)^{[1]}\otimes L(1,0)^{[2]}$ (see the proof of Theorem 2); here, the number of tensor factors is one more than the number $r$. So this question seems bound up with a related question, that is:

\begin{question}Does there exist a constant $\delta=\delta(m)$ with the property: 

{\it Let $G$ be a simple algebraic group defined over a field of characteristic $p>h$; and suppose $L(\lambda)$ and $L(\mu)$ are simple $G$-modules with Steinberg tensor decompositions $L(\lambda_0)\otimes \dots \otimes L(\lambda_r)^{[r]}$ and $L(\mu_0)\otimes\dots\otimes L(\mu_s)^{[s]}$, respectively. Set $I_{\lambda,\mu}=\#\{\lambda_i\neq \mu_i\}$, where we take $\lambda_i=0$ for $i>r$ and $\mu_i=0$ for $i>s$. Then if $\Ext^m_G(L(\lambda),L(\mu))\neq 0$, one has that $I_{\lambda,\mu}\leq\delta(m)$.} ?\end{question}

The above is ambitious, perhaps. If we let $\delta=\delta(m,\Phi)$ depend on the root system $\Phi$ of $G$ then by \cite{Cli79}, we can take $\delta(1,A_1)=2$; by \cite[4.2.3]{Yeh82} we can use $\delta(1,A_2)=2$; from \cite[5.1]{Ye90} we take $\delta(1,B_2)=2$ and from \cite[\S5, Theorem]{Ye93} one also uses $\delta(1,G_2)=2$. Indeed if we take also $p\geq 3h-3$, \cite[Corollary 2.4]{BNP02} says that one can use $\delta(1)=2$, which is of course also the minimum one can take. Might one possibly use $\delta(1)=2$ for all root systems and {\it all} characteristics?

Notice that we need to increase $\delta$ for $2$-cohomology. We have $\delta(2,A_1)\geq 4$ from \cite[Theorem 1]{Ste10} and $\delta(2,A_2)\geq 4$ from Theorem 1 above should they exist. Perhaps $\delta(2)=4$ would suffice?

{\footnotesize\bibliographystyle{amsalpha}
\bibliography{stewart.bib}

David Stewart, New College, Oxford, OX1 3BN, UK. \\ Email: dis20@cantab.net}
\end{document}